\documentclass[11pt,a4paper]{amsart}
\usepackage{amsmath,amssymb,amsthm}
\usepackage[utf8]{inputenc}
\usepackage{graphicx}
\usepackage{color}

\swapnumbers

\theoremstyle{plain}
\newtheorem{thm}[equation]{Theorem}

\newtheorem{pro}[equation]{Proposition}
\newtheorem{cor}[equation]{Corollary}
\theoremstyle{definition}

\newtheorem{alg}[equation]{Algorithm}
\newtheorem{ass}[equation]{Assumption}

\theoremstyle{remark}
\newtheorem{rem}[equation]{Remark}

\numberwithin{equation}{section}

\newcommand{\1}{\mathbf{1}}
\renewcommand{\d}{\mathrm{d}}
\newcommand{\E}{\mathbb{E}}

\renewcommand{\P}{\mathbb{P}}

\renewcommand{\phi}{\varphi}
\newcommand{\R}{\mathbb{R}}

\newcommand{\dist}{\mathrm{dist}}

\title[WOS for Helmholtz and Yukawa Equations]{Walk on Spheres Algorithm for Helmholtz and Yukawa Equations via Duffin Correspondence}

\author{Xuxin Yang}
\address{Xuxin Yang\\
Hunan First Normal University\\
Department of Mathematics\\
Changsha, Hunan 410205\\
P.R. of CHINA}

\author{Antti Rasila}
\address{Antti Rasila\\
Aalto University\\
School of Science\\
Department of Mathematics and Systems Analysis\\
P.O.Box 1100\\
FIN-00076 Aalto\\
FINLAND}

\author{Tommi Sottinen}
\address{Tommi Sottinen\\
University of Vaasa\\
Faculty of Technology \\
Department of Mathematics and Statistics\\
P.O.Box 700\\
FIN-65101 Vaasa\\
FINLAND}

\thanks{
This work is supported by the NNSF of China (No. 11571088), Hunan Provincial Natural Science Foundation of China
(14JJ7083), a Key Project Supported by Scientific Research Fund of Hunan Provincial Education Department (14A028), 
Scientific Research Fund of Hunan Provincial Education Department (14C0253), the Aid Program for Science and Technology Innovative
Research Team in Higher Educational Institutions of Hunan Province. A. Rasila was partially supported by Academy of Finland (No. 289576).
T. Sottinen was partially funded by the Finnish Cultural Foundation (National Foundations' Professor Pool).  This work was done during T. Sottinen's research visit to Hunan First Normal University, Changsha, PRC.  T. Sottinen wishes to thank them for their hospitality.}



\begin{document}

\begin{abstract}
We show that a constant-potential time-independent Schr\"odinger equation with Dirichlet boundary data can be reformulated as a Laplace equation with Dirichlet boundary data.  With this reformulation, which we call the Duffin correspondence, we provide a classical Walk On Spheres (WOS) algorithm for Monte Carlo simulation of the solutions of the boundary value problem.  We compare the obtained Duffin WOS algorithm with existing modified WOS algorithms.  

\end{abstract}

\maketitle

\section{Introduction}

We consider stochastic simulation (or Monte Carlo simulation) of the solution of a Dirichlet boundary value problem of a time-independent Schr\"odinger equation with constant potential.  For positive potentials this equation is called the Yukawa equation or the linearized Poisson--Boltzmann equation. For zero potential the equation is known as the Laplace equation. For negative potential it is known as the Helmholtz equation.  

The connection between the Dirichlet boundary value problems and the Brownian motion date back to Kakutani \cite{kakutani}, who provided a stochastic representation of the solution of the Laplace equation with Dirichlet boundary conditions in terms of the exit locations of a Brownian motion.  Later, this connection has been extended for the Schr\"odinger equation, cf. \cite{chung-zhao} and references therein.  These stochastic representations for constant non-zero potentials include the exit time of the Brownian motion in addition to the exit location; and in the case of non-constant potential, the stochastic representation depends on the entire path of the Brownian motion up to the time it exits the domain. 

Stochastic representations provide Monte Carlo simulation methods for the solutions of the boundary value problems.  These methods are especially attractive in high dimensions where deterministic methods are typically extremely costly.
The obvious idea to use the stochastic representations is to simulate Brownian particles on the domain with a fine time-mesh.  This is, however, very costly computationally.  In order to avoid the simulation of the trajectory of the Brownian particles very precisely, different algorithms have been proposed. In this paper, we consider different Walk On Spheres algorithms that simulate the Brownian motion only on successive spheres in the domain. 

If one only needs to simulate the exit location, and not the exit time, of the Brownian motion, one can use the very efficient classical Walk On Spheres (WOS) algorithm due to Muller \cite{muller}.  Unfortunately, the stochastic representation involving only the exit locations of Brownian motion corresponds precisely the Laplace equation.
This is our motivation to transform the constant-potential Schr\"odinger equation into a Laplace equation: to make the classical WOS algorithm applicable.  Also, this transformation should be of interest in its own right.

The transformation, the so-called Duffin correspondence, removes the constant potential in the Schr\"odinger equation with the cost of adding one extra dimension to the boundary value problem. The idea of the correspondence is due to Duffin \cite{duffin1}, where the correspondence was used for the Yukawa equation, i.e. for the case of positive constant potential, on the plane.  This was later extended to general Euclidean spaces in \cite{panharmonic}.  In this paper, we extend the Duffin correspondence to cover also the Helmholtz case, i.e. negative constant potential, in general Euclidean spaces.  

Let us note that there are already efficient modified WOS algorithms that simulate the exit time (or its Laplace transform) and the exit location of the Brownian motion that can be used for the constant-potential Schr\"odinger equation studied here. Indeed, such stochastic simulation algorithms have been studied excessively; cf. \cite{woms,woms2,elepov-mikhailov,hwang-mascagni,hwang-mascagni-given}, just to mention few. 
Basically, in a modified WOS algorithm that simulates the Laplace transform of the exit time and the exit location of the Brownian particle, one needs to keep track of a multiplicative weight for the simulated Brownian particle.  We call this algorithm the Weighted Walk On Spheres (WWOS) and recall it in Section \ref{sect:wos-algorithms}.  If the constant potential is negative, then the weight of the Brownian particle can be reinterpreted as independent exponential killing of the particle.  We call this algorithm the Killing Walk On Spheres (KWOS) and recall it in Section \ref{sect:wos-algorithms}. 
Therefore, we admit that there are already efficient algorithms for the problem studied here.  However, it is our opinion that our Duffin correspondence WOS algorithm (DWOS) is of comparable efficiency to the known modified WOS algorithms and has the advantage of being both simpler to comprehend and easier to implement than the modified WOS algorithms known so far.  Indeed, DWOS algorithm is simply the classical WOS algorithm with an added dimension and multiplicatively modified boundary data: there is no need to keep track of any weight or killing.  Moreover, if the WOS algorithm is already implemented, the DWOS algorithm does not need implementation: it is simply the WOS algorithm with different input. 

The rest of the paper is organized as follows. In Section \ref{sect:prelim}, we lay the setting and recall the connection between the Dirichlet boundary value problems and the Brownian motion. In Section \ref{sect:duffin}, we prove our main result, the Duffin correspondence, and the stochastic representation of the solutions of the constant-potential Schr\"odinger equation without the stopping time distribution.  Section \ref{sect:wos-algorithms} is devoted to the different WOS algorithms and their implementations. In Section \ref{sect:comparisons}, we provide examples and comparisons of the different WOS algorithms.  Finally, in Section \ref{sect:conclusions} we draw some some conclusions on the performance of the DWOS algorithm.

\section{Preliminaries}\label{sect:prelim}

Let $D\subset\R^n$ be a domain (i.e., open and connected) satisfying Assumption \ref{ass} below. Let $x=(x_1,\ldots,x_n)$. Denote by
$$
\Delta_x = \sum_{i=1}^n \frac{\partial^2}{\partial x_i^2}
$$
the Laplacian with respect to the variable $x\in \R^n$.

We consider the Dirichlet-type boundary value problem of the Schr\"odinger equation with constant potential $\lambda\in\R$:
\begin{equation}\label{eq:helmholtz-yukawa}
\left\{\begin{array}{rcll}
\frac12 \Delta_x u(x) -\lambda u(x) &=& 0    &\mbox{on } x\in D, \\
                               u(y) &=& f(y) &\mbox{on } y\in \partial D.
\end{array}\right.
\end{equation}
Here $f$ is (continuous and) bounded on $\partial D$. The case $\lambda>0$ corresponds to the Yukawa equation, or the linearized Poisson--Boltzmann equation. The case $\lambda=0$ is the Laplace equation. The case $\lambda<0$ is the Helmholtz equation.

The required regularity conditions for the domain $D$ of a Helmholtz--Laplace--Yukawa type Dirichlet boundary value problem \eqref{eq:helmholtz-yukawa} to admit a unique bounded (strong) solution are best expressed by using probabilistic tools and the Brownian motion:

Let $W$ be a standard $n$-dimensional Brownian motion. Let $\tau_D$ be the first exit time of the Brownian motion $W$ from the domain $D$, i.e., 
$$
\tau_D = \inf\{ t>0\,;\, W_t \not\in D\}.
$$
Here, as always, use the normal convention that 
$$
\inf\varnothing = +\infty.
$$

The following assumptions on the domain $D$ are always in force, although not explicitly stated later:

\begin{ass}\label{ass}
\mbox{} 
\begin{enumerate}
\item
The domain $D$ is \emph{Wiener regular}, i.e.,
\begin{equation*}\label{eq:regular}
\P^y[\tau_D=0] \quad\mbox{for all } y\in\partial D.
\end{equation*}
\item
The domain $D$ is \emph{Wiener small}, i.e.,
\begin{equation*}\label{eq:small}
\P^x[\tau_D<\infty] = 1 \quad\mbox{for all } x\in D.
\end{equation*}
\item
Finally, we assume, by using the terminology of Chung and Zhao \cite{chung-zhao},  that the domain $D$ is \emph{gaugeable}, i.e.,
\begin{equation*}\label{eq:gauge}
\sup_{x\in D} \E^x\left[e^{-\lambda \tau_D}\right] < \infty.
\end{equation*}
\end{enumerate}
\end{ass}

\begin{rem}\label{rem:ass}
\mbox{}
\begin{enumerate}
\item
All domains with piecewise $C^1$ boundary are Wiener regular.
\item
If any projection of the domain $D$ on any subspace $\R^{n'}$, $n'\le n$ is bounded, then $D$ is Wiener bounded.
\item
For $\lambda\ge 0$, the gauge condition \ref{ass}(iii) is vacuous; for $\lambda<0$ it is essential. 

Actually, it follows from \cite[Theorem 4.19]{chung-zhao} that the gauge condition \ref{ass}(iii) is satisfied if and only if $\lambda>\lambda_1(D)$, where $\lambda_1(D)$ is the principal Dirichlet eigenvalue of the negative half-Laplacian, i.e., for $\lambda>\lambda_1(D)$ the boundary value problem 
\begin{equation*}\label{eq:dirichlet-eigenvalue}
\left\{\begin{array}{rcll}
\frac12 \Delta_x u(x) -\lambda u(x) &=& 0 &\mbox{on } x\in D, \\
                                    u(y) &=& 0 &\mbox{on } y\in \partial D.
\end{array}\right.
\end{equation*}
admits only the trivial solution $u(x)\equiv 0$.
By the Rayleigh--Faber--Krahn inequality \cite{krahn}
\begin{equation}\label{eq:rfk}
\lambda_1(D) \ge -\frac12 
m(D)^{-2/n} \left(\frac{\pi^{n/2}}{\Gamma(n/2+1)}\right)^{2/n} j_{n/2-1,1}^2, 
\end{equation}
where $j_{\nu,1}$ is the smallest positive zero of the Bessel function 
\begin{equation}\label{eq:bessel}
J_{\nu}(x) = \sum_{i=0}^\infty 
\frac{(-1)^i}{\Gamma(i+1)\Gamma(i+\nu+1)}\left(\frac{x}{2}\right)^{\nu+2i}.
\end{equation} 
Equality is attained in \eqref{eq:rfk} if and only if $D$ is a ball; e.g. for the unit ball we have
$$
\lambda_1(B_n(0,1)) = -\frac12 j_{n/2-1,1}^2.
$$ 
This relation is pronounced later in formula \eqref{eq:psi-neg}. Numerical approximations of $\lambda_1(B_n(0,1))$ are give in Table \ref{tab:rayleigh}.
\begin{table}[h]
\begin{tabular}{c|c}
$n$ & $\lambda_1(B_n(0,1))$\\ \hline
 1 & -6.283185\\
 2 & -47.46935\\
 3 & -461.7912\\
 4 & -4544.658\\
 5 & -36316.03
 \end{tabular}

 \caption{Numerical approximations of the Rayleigh--Faber--Krahn constant $\lambda_1(B_n(0,1))$ for $n=1,2,\ldots,5$.}\label{tab:rayleigh}
\end{table}
\end{enumerate}
\end{rem}

The following stochastic representation of the bounded solutions to the boundary value problem \eqref{eq:helmholtz-yukawa} is well-known.  See, e.g., \cite[Chapter 4]{chung-zhao} or \cite[Chapter 4]{durrett}:

\begin{pro}[Stochastic representation]\label{pro:kakutani}
The boundary value problem \eqref{eq:helmholtz-yukawa} admits a unique bounded solution given by
$$
u(x) = \E^x\left[ e^{-\lambda\tau_D}f\left(W_{\tau_D}\right)\right].
$$
\end{pro}

\section{Duffin correspondence}\label{sect:duffin}

Associated with the Helmholtz or the Yukawa boundary value problem \eqref{eq:helmholtz-yukawa} with $\lambda\ne 0$, there is a classical Laplace boundary value problem with $\lambda=0$ on an extended domain $\bar D(\lambda)$:  Indeed, define $g(\lambda;\cdot)\colon\R\to\R$ by
\begin{equation*}\label{eq:g}
g(\lambda;x') =
\left\{\begin{array}{rl}
\cos\left(\sqrt{2\lambda}\, x'\right), & \mbox{for } \lambda>0, \\
\cosh\left(\sqrt{-2\lambda}\, x'\right), & \mbox{for } \lambda<0.
\end{array}\right.
\end{equation*}  
Set
\begin{equation*}\label{eq:I}
I(\lambda) = \left\{
\begin{array}{rl}
\left(-\frac{\pi}{2\sqrt{2\lambda}}, \frac{\pi}{2\sqrt{2\lambda}}\right), & \mbox{for } \lambda > 0, \\
\R, & \mbox{for } \lambda <0,
\end{array}
\right.
\end{equation*}
and
\begin{equation*}\label{eq:D-bar}
\bar D(\lambda) = D\times I(\lambda).
\end{equation*}
Finally, denote $\bar x =(x,x')\in \R^{n}\times \R$, and set 
\begin{eqnarray*}\label{u-bar}
\bar u(\lambda;\bar x) &=& u(x)g(\lambda;x'), \\
\label{eq:f-bar}
\bar f(\lambda;\bar y) &=& f(y)g(\lambda;y'). 
\end{eqnarray*}
With this notation, consider the following family of Laplace boundary value problems indexed by $\lambda\ne 0$:
\begin{equation}\label{eq:laplace}
\left\{\begin{array}{rcll}
\frac12 \Delta_{\bar x} \bar u(\lambda;\bar x) &=& 0    &\mbox{on } \bar x\in \bar D(\lambda), \\
                              \bar u(\lambda;\bar y) &=& \bar f(\lambda; \bar y) &\mbox{on } \bar y\in \partial \bar D(\lambda).
\end{array}\right.
\end{equation}

\begin{thm}[Duffin correspondence]\label{thm:duffin}
Let $\lambda\ne 0$ be fixed. Then $u$ is the unique bounded solution to the Helmholtz or Yukawa boundary value problem \eqref{eq:helmholtz-yukawa} if and only if $\bar u$ is the unique bounded solution to the Laplace equation \eqref{eq:laplace}.
\end{thm}

\begin{proof}
The Yukawa case, $\lambda>0$, was shown in \cite{panharmonic} (for general $n$) and in the original paper by Duffin \cite{duffin1} (for $n=2$).

Let us consider the Helmholtz case, $\lambda<0$. The proof that 
$$
\frac12 \Delta_{\bar x} \bar u(\lambda;\bar x) = 0
$$ 
if and only if
$$
\frac12 \Delta_x u(x) -\lambda u(x) = 0
$$
is straightforward and can be done exactly as in the Yukawa case.  Also, it is straightforward to see that $D$ satisfies assumptions \ref{ass}(i) and \ref{ass}(ii) if and only if $\bar D(\lambda)$ satisfies assumptions \ref{ass}(i) and \ref{ass}(ii).  The essential difference to the Yukawa case is that now $\bar D(\lambda)$ is unbounded in the $(n+1)^\mathrm{th}$ co-ordinate.  Consequently, the boundary data $\bar f(\lambda;\cdot)$ is not bounded.  This is where we need the gauge condition \ref{ass}(iii).  Indeed, we can approximate the solution in $D\times [-M,M]$, and the result follows from the stochastic representation \ref{pro:kakutani} and the dominated convergence theorem by letting $M\to\infty$.
\end{proof}

Let $\bar W=(W,W')$ be an $(n+1)$-dimensional Brownian motion and let
$$
\tau_{I(\lambda)} =\inf\{ t>0\,;\, W'_t \not\in I(\lambda)\}.
$$
Then, by using the Duffin correspondence Theorem \ref{thm:duffin}, we obtain a stochastic representation for the solution of the boundary value problem \eqref{eq:helmholtz-yukawa} in terms of the exit location distribution only, i.e., a stochastic solution that is independent of the exit time distribution.

\begin{cor}[Stochastic representation without exit time]\label{cor:kakutani}
The boundary value problem \eqref{eq:helmholtz-yukawa} admits a solution
$$
u(x) = \E^{x,0}\left[f\left(W_{\tau_D}\right)g\left(\lambda;W'_{\tau_D}\right)\,;\, \tau_{I(\lambda)}>\tau_D\right].
$$
\end{cor}

\section{Walk on spheres algorithms}\label{sect:wos-algorithms}

If one wants to simulate the solutions of \eqref{eq:helmholtz-yukawa} directly by using the stochastic representation of Proposition \ref{pro:kakutani}, then one must simulate both the exit time and the exit position of the Brownian particle moving in the domain $D$. This can be done e.g. by using a weighted walk on spheres (WWOS), or --- in the Yukawa setting --- killing walk on spheres (KWOS).  (It is also possible to use the recent walk on moving spheres (WOMS) algorithm developed in \cite{woms,woms2}.)  
However, by using the Duffin correspondence of Theorem \ref{thm:duffin} one only needs to simulate the exit position of the Brownian particle in the extended domain $\bar D(\lambda)$.  This allows one to use the classical Muller's \cite{muller} walk on spheres (WOS) algorithm. We call this extension the Duffin Walk On Spheres (DWOS) algorithm.

In this section, we will explain in detail these walk on spheres algorithms --- DWOS, WWOS, KWOS --- for the Yukawa--Laplace--Helmholtz equations.  Next, in Section \ref{sect:comparisons} we will provide examples and comparisons of these algorithms. 

We start with the most elementary problem: how to generate uniform random variables on spheres.

\subsection*{Generating uniform distribution on spheres}

In all the variants of the walk on spheres algorithms presented here, one needs to simulate random variables that are uniformly distributed on the $(n-1)$-dimensional unit spheres $\partial B(0,1) = \partial B_n(0,1)$.  Most mathematical software has this functionality built in.  In case your favorite software does not have this functionality, Algorithm \ref{alg:uniform-sphere} below explains how to generate uniform distribution on the $(n-1)$-dimensional unit spheres by using independent normally distributed random variables. Generating normally distributed random variables should be built-in in almost all mathematical software.

\begin{alg}[Uniform distribution on the unit sphere]\label{alg:uniform-sphere}
Generate $X_1,\ldots,X_n$ independent standard normal random variables. Set
$S^2=X_1^2+\ldots+X_n^2$. Then $U = (X_1/S, \ldots, X_n/S)$ is uniformly distributed on the $(n-1)$-dimensional unit sphere.  
\end{alg}

\subsection*{Duffin Walk On Spheres}

For the DWOS Algorithm \ref{alg:dwos} below, we use the stochastic representation of Corollary \ref{cor:kakutani}.  

Recall that $\bar W = (W,W')$ is the $(n+1)$-dimensional Brownian motion, $\P^{\bar x}$ and $\E^{\bar x}$ are the probability law and expectation, respectively, under which $\bar W_0 = \bar x$. Under the probability law $\P^{\bar x}$, denote
\begin{eqnarray*}
\tau_x &=& \inf\left\{ t>0\,;\, W_t\not\in D\right\}, \\
\tau'_{x'} &=& \inf\left\{ t>0\,;\, W'_t \not\in I(\lambda)\right\}, \\
\bar\tau_{\bar x} &=& \inf\left\{ t>0\,;\, \bar W \not\in \bar D(\lambda)\right\}.
\end{eqnarray*}

\begin{rem}\label{rem:exit-first}
$\tau_x$ and $\tau'_{x'}$ are independent, and
$
\bar \tau_{\bar x} = \min\left(\tau_x,\tau'_{x'}\right).
$
\end{rem}

For $x\in D$, denote $\bar x = (x',0)\in \bar D(\lambda)$. 

The stochastic approximation for $u(x)$ is 
\begin{eqnarray}\label{eq:u-wos}
\hat u_K(x) &=& 
\frac{1}{K} \sum_{k=1}^K \bar f\left(\lambda\,;\, \bar w^k_{\bar x}\big(\bar \tau_{\bar x}^k\big)\right) \\
&=&
\frac{1}{K}\sum_{k=1}^K f\left(w_x^k(\bar\tau^k_{\bar x})\right) g\left(\lambda\,;\, (w')^k_0(\bar\tau^k_{\bar x})\right). \nonumber
\end{eqnarray}
Here$\bar\tau^k_x$ is the termination-step of the trajectory of the each individual Brownian particle $\bar w_{\bar x}^k$ starting from $\bar x=(x,0)\in \bar D(\lambda)=D\times I(\lambda)$: 
\begin{eqnarray*}
\bar w_{\bar x}^k
&=&
\left( \bar w_{\bar x}^k(j)\right)_{j=0}^{\bar\tau_{\bar x}^k} \\
&=& \left(w_x^k,(w'_0)^k\right)_{j=0}^{\bar\tau_{\bar x}^k}  \\
&=& \left(\big(x,0\big),\big(w_x^k(1),(w'_0)^k(1)\big),\ldots,
\big(w_x^k(\bar\tau^k_{\bar x}),(w'_0)^k(\bar \tau^k_{\bar x})\big)\right).
\end{eqnarray*}

For each $k=1,\ldots, K$, and $\bar x=(x,0)$, $x\in D$, the individual particle exit locations $\bar w_{\bar x}(\bar\tau^k_{\bar x})$ for the approximating sum \eqref{eq:u-wos} are generated by Duffin Walk On Spheres (DWOS) Algorithm \ref{alg:dwos} below:

\begin{alg}[DWOS]\label{alg:dwos}
Fix a small parameter $\varepsilon>0$.
\begin{enumerate}
\item
{Initialize}:  $\bar w_{\bar x}(0)=(w_x(0), (w')_0(0)) = (x,0)$. 
\item
{While} $\dist(\bar w_{\bar x}(j), \partial \bar D(\lambda)) > \varepsilon$:
\begin{enumerate}
\item
Set $r(j) = \dist(\bar w_{\bar x}(j),\partial \bar D(\lambda))$.
\item
Sample $\xi(j)$ independently from the unit sphere $\partial B_{n+1}(0,1)$
(by using Algorithm \ref{alg:uniform-sphere}). 
\item
Set $\bar w_{\bar x}(j+1) = \bar w_{\bar x}(j) + r(j)\xi(j)$.
\end{enumerate} 
\item
{When} $\dist(\bar w_{\bar x}(j), \partial \bar D(\lambda)) \le \varepsilon$:
\item
Set $\mathrm{pr}\, \bar w_{\bar x}(j)$ to be the orthogonal projection of $\bar w_{\bar x}(j)$ to $\partial\bar D(\lambda)$.
\item
{Return} $\mathrm{pr}\, \bar w_{\bar x}(j)$.
\end{enumerate}
\end{alg}

\subsection*{Weighted Walk On Spheres}

Suppose we want use the WOS algorithm directly without the Duffin correspondence.  To do this, we must estimate the term $e^{-\lambda\tau_D}$ in Proposition \ref{pro:kakutani}.
Suppose the WOS algorithm takes $T$ steps to hit the boundary with balls of radii $r_1, r_2,\ldots, r_T$. The weighted walk on spheres (WWOS) algorithm \ref{alg:wwos} below is based on the fact that the term $e^{-\lambda\tau_D}$ can be decomposed into independent terms
$$
e^{-\lambda\tau_D} = e^{-\lambda\tau_{r_1}}e^{-\lambda\tau_{r_2}}\cdots e^{-\lambda\tau_{r_T}},
$$
where the $\tau_{r_j}$'s are the exit times of the Brownian motion from balls of radius $r_j$, and these exit times are also independent of the exit locations from the ball.  Consequently, at each step $j$ of the WOS algorithm, the Brownian particle gains (or loses) an independent multiplicative weight that is given by
$$
\E\left[e^{-\lambda\tau_{r_j}}\right].
$$ 
Here $\E=\E^0$.

By using the $1/2$-self-similarity of the Brownian motion we see that 
\begin{equation}\label{eq:def-psi}
\E\left[e^{-\lambda\tau_{r}}\right] =
\E\left[e^{-\lambda r^2 \tau_{1}}\right] 
= \psi(\lambda r^2).
\end{equation}

For $\mu>0$ the function $\psi$ is well-known, cf. Wendel \cite{wendel}:
\begin{equation}\label{eq:psi-pos}
\psi(\mu) =
\left\{\begin{array}{rl}
\frac{\mu^\nu}{2^\nu\Gamma(\nu+1)I_\nu(\mu)}, & n=2\nu-2\ge 2, \\
\frac{1}{\cosh(\sqrt{2\lambda})}, & n=1.
\end{array}\right.
\end{equation}
Here $I_\nu$ is the modified Bessel function
$$
I_\nu(x) = \sum_{i=0}^\infty \frac{1}{i!\Gamma(i+\nu+1)}
\left(\frac{x}{2}\right)^{2i+\nu}.
$$

For $\mu<0$, as far as we know, no simple formula for $\psi(\mu)$ is known. However, the distribution function of $\tau_1$ is well-known, cf. Kent \cite{kent} or Ciesielski and Taylor \cite{ciesielski-taylor}:
\begin{equation}\label{eq:hitting-distribution}
\P\left[\tau_1\le t\right] = 1 -
\frac{1}{2^{\nu-1}\Gamma(\nu+1)}\sum_{i=1}^{\infty}
\frac{j_{\nu,i}^{\nu-1}}{J_{\nu+1}(j_{\nu,i})}\exp\left\{-\frac12 j_{\nu,i}^2 \,t\right\}.
\end{equation}
Here $\P=\P^0$, $\nu=n/2-1$, and $j_{\nu,i}$'s are the positive zeros of the Bessel function $J_\nu$ given by \eqref{eq:bessel} in the increasing order.
Consequently, by simple substitution to \eqref{eq:hitting-distribution} combined with change of differentiation and summation,
\begin{eqnarray*} 
\lefteqn{\psi(\mu)} \\
&=&
-\int_0^\infty e^{-\mu t}\,\d\left[\frac{1}{2^{\nu-1}\Gamma(\nu+1)}\sum_{i=1}^{\infty}
\frac{j_{\nu,i}^{\nu-1}}{J_{\nu+1}(j_{\nu,i})}
\exp\left\{-\frac{1}{2} j_{\nu,i}^2\, t\right\}
\right] 
\nonumber \\ 
&=&
\frac{-1}{2^{\nu-1}\Gamma(\nu+1)}\sum_{i=1}^{\infty}
\frac{j_{\nu,i}^{\nu-1}}{J_{\nu+1}(j_{\nu,i})}\int_0^\infty e^{-\mu t}\,\d\left[
\exp\left\{-\frac{1}{2} j_{\nu,i}^2\, t\right\}
\right] \nonumber \\
&=&
\frac{1}{2^{\nu}\Gamma(\nu+1)}\sum_{i=1}^{\infty}
\frac{j_{\nu,i}^{\nu+1}}{J_{\nu+1}(j_{\nu,i})}\int_0^\infty \exp\left\{-\left(\mu+\frac{1}{2} j_{\nu,i}^2\right) t\right\}
\, \d t.
\label{eq:psi}
\end{eqnarray*}
By the gauge condition $\mu>-1/2 j_{\nu,1}$. Therefore,
\begin{eqnarray*}
\lefteqn{\psi(\mu)} \\
&=&
\frac{1}{2^{\nu}\Gamma(\nu+1)}\sum_{i=1}^{\infty}
\frac{j_{\nu,i}^{\nu+1}}{J_{\nu+1}(j_{\nu,i})}\int_0^\infty \exp\left\{-\left(\mu+\frac{1}{2} j_{\nu,i}^2\right) t\right\}
\, \d t \\
&=&
\frac{-1}{2^{\nu}\Gamma(\nu+1)}\sum_{i=1}^{\infty}
\frac{j_{\nu,i}^{\nu+1}}{J_{\nu+1}(j_{\nu,i})\left(\mu+\frac12 j_{\nu,i}^2\right)}\left[
\exp\left\{-\left(\mu+\frac{1}{2} j_{\nu,i}^2\right) t\right\}
\right]_{t=0}^{\infty} \\
&=& 
\frac{1}{2^{\nu}\Gamma(\nu+1)}\sum_{i=1}^{\infty}
\frac{j_{\nu,i}^{\nu+1}}{J_{\nu+1}(j_{\nu,i})\left(\mu+\frac12 j_{\nu,i}^2\right)}.
\end{eqnarray*} 
Consequently, we have obtained
\begin{equation}\label{eq:psi-neg}
\psi(\mu) = 
\frac{1}{2^{\nu}\Gamma(\nu+1)}\sum_{i=1}^{\infty}
\frac{j_{\nu,i}^{\nu+1}}{J_{\nu+1}(j_{\nu,i})\left(\mu+\frac12 j_{\nu,i}^2\right)},
\end{equation} 
a formula that is true for all $\mu>-\frac12 j_{\nu,1}^2$, i.e., whenever the gauge condition \ref{ass}(iii) for the unit ball holds.

\begin{rem}
While it is possible to directly use the formula \eqref{eq:psi-neg} in numerical computations, our experimental results show that this is very inefficient as the sum in the formula \eqref{eq:psi-neg} converges very slowly and several thousand terms are needed for workable accuracy. There are at least two reasonable approaches to address this shortcoming: 
\begin{enumerate}
\item 
One may use the formula \eqref{eq:hitting-distribution} to tabulate values of the function $\P\left[\tau_1\le t\right]$ for given values of $t_j\in(0,\infty)$, where $t_j<t_{j+1}$ and $j=1,2,\ldots,N$, and then use the identity \eqref{eq:def-psi} and the elementary properties of the expected value to obtain the approximation
$$
\psi(\lambda r^2) \approx \sum_{j=1}^{N-1} e^{-\lambda r^2(t_j+t_{j+1})/2}\left(\P\left[\tau_1\le t_{j+1}\right]-\P\left[\tau_1\le t_j\right]\right).
$$
\item 
One may use the formula \eqref{eq:psi-neg} to tabulate values of the function $\psi(\lambda r^2)$ for various values of $r>0$, and then use the table and appropriate interpolation to approximate the function $\psi(\lambda r^2)$ for an arbitrary value of $r$.
\end{enumerate}
The first approach was used in examples of Section \ref{sect:comparisons}. Approximations for function $\psi(\lambda r^2)$, for various values of $\lambda, r$ and $n$ are illustrated in Figure \ref{fig:psifun} (cf. the figure in \cite{panharmonic} for an illustration of $\psi$ for positive values of $\lambda$).
\end{rem}

\begin{figure}[h]

\includegraphics[width=5.5cm]{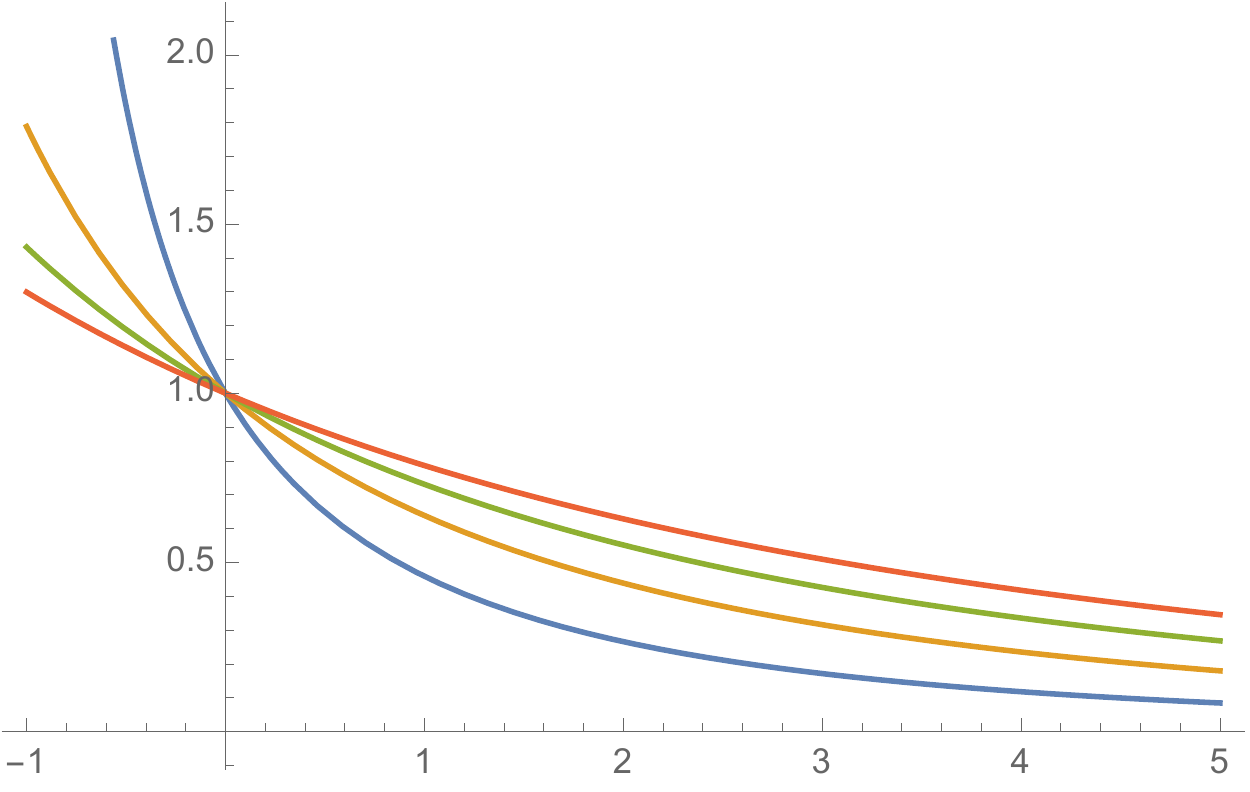}\qquad
\includegraphics[width=5.5cm]{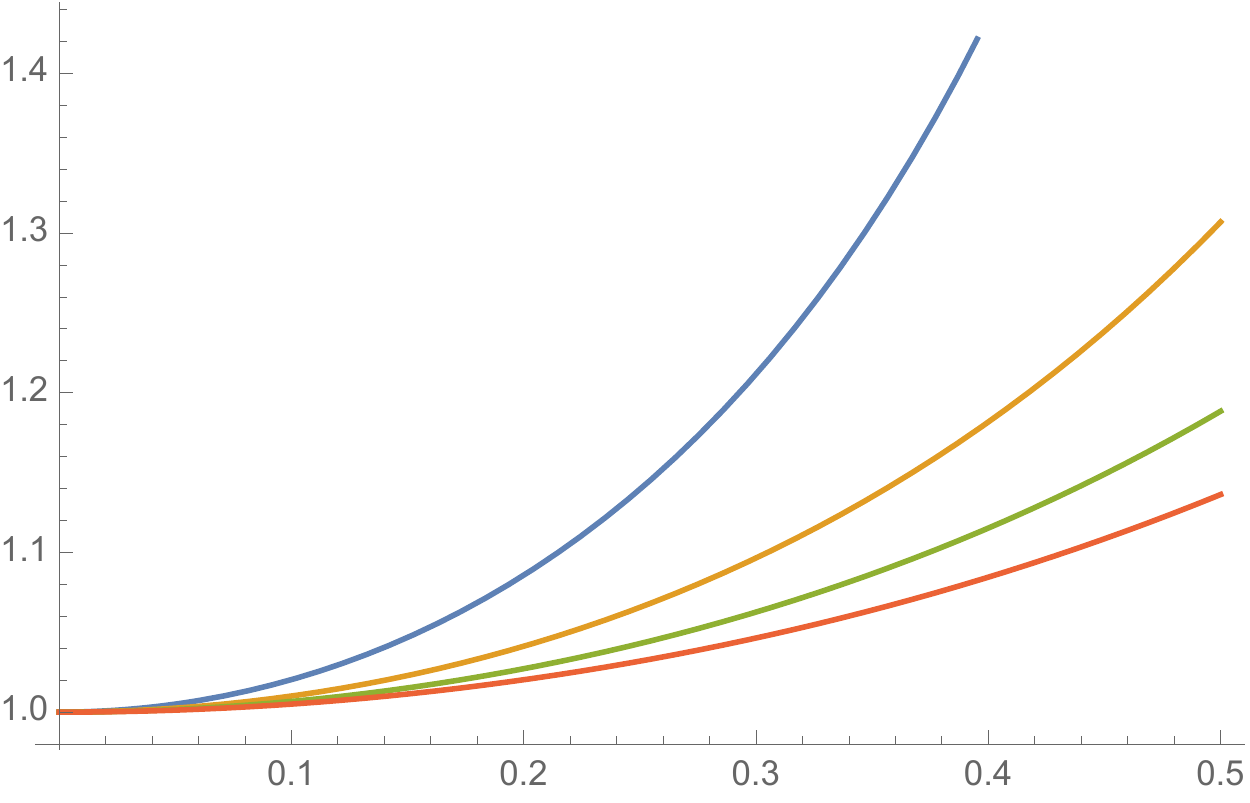}

\caption{ The function $\psi(\lambda)$ for $\lambda\in (-1,5)$ for $n=1,2,3,4$, where $n=1$ on the top (left). The function $\psi(-2r^2)$, $r>0$, where the dimension are $n=1,2,3,4$, where $n=1$ on the top (right). }\label{fig:psifun}
\end{figure}

\begin{rem}
Note that \eqref{eq:psi} is the extension of the $\psi$-function, that appears in \cite{panharmonic}, but with different parametrization.
\end{rem}

Now we are ready to give the Weighted Walk On Spheres (WWOS) algorithm:
The approximation for $u(x)$ is, by Proposition \ref{pro:kakutani}, 
\begin{equation}\label{eq:u-wwos}
\hat u_K(x) = \frac{1}{K}\sum_{k=1}^K c_x^k(\lambda) f\left(w_x^k(\tau^k_x)\right).
\end{equation}
Here$\tau^k_x$ is the exit time for the each individual particle and
$c_x^k(\lambda) = e^{-\lambda\tau^k_x}$.

The individual particle exit locations $w_{x}(\tau_x)$ and weights $c_x^k$ are generated by Algorithm \ref{alg:dwos} below:

\begin{alg}[WWOS]\label{alg:wwos}
Fix a small parameter $\varepsilon>0$.
\begin{enumerate}
\item
Initialize:  $w_x(0)=x$, $c_x(0)=0$ 
\item
While $\dist(w_{x}(j), D) > \varepsilon$:
\begin{enumerate}
\item
Set $r(j) = \dist(w_{x}(j),\partial D)$.
\item
Sample $\xi(j)$ independently from the unit sphere $\partial B_{n}(0,1)$
(by using Algorithm \ref{alg:uniform-sphere}). 
\item
Set $w_{x}(j+1) = w_{x}(j) + r(j)\xi(j)$ and
$c_x(j+1)=c_x(j)\psi(\lambda r^2)$. Here $\psi$ is given by \eqref{eq:psi-pos} for $\lambda>0$ and \eqref{eq:psi-neg} for $\lambda<0$.
\end{enumerate} 
\item
When $\dist(w_{x}(j), \partial D) \le \varepsilon$:
\item
Set $\mathrm{pr}\, w_{x}(j)$ to be the orthogonal projection of $w_{x}(j)$ to $\partial D$.
\item
Return $\mathrm{pr}\, w_{x}(j)$ and $c_x(j)$.
\end{enumerate}
\end{alg}

\subsection*{Killing Walk On Spheres}

For the Yukawa case $\lambda>0$, the weight loss $e^{-\lambda \tau_D}$ of the particle can be interpreted as independent exponential killing of the particle.  See \cite{elepov-mikhailov}, \cite{panharmonic} or \cite{yrs} for details.
Consequently, the WWOS algorithm \ref{alg:wwos} can be reinterpreted as Killing Walk On Spheres (KWOS).

Our estimator for $u(x)$ is
\begin{equation}\label{eq:simu-killing-kakutani}
\hat u_K (x) = \frac{1}{K}\sum_{k\in K^*(\lambda)} f(w_x^k(\tau^k_x)), 
\end{equation}
where $w_x^k$, $k=1,\ldots, K$ are independent simulations of the trajectories Brownian particles starting from point $x$, and the set $K^*(\lambda)\subset\{1,\ldots,K\}$ contains the particles that are not killed; $\tau^k_x$ is the termination-step time of the algorithm.  The individual particles are generated by Algorithm \ref{alg:kwos} below.

\begin{alg}[KWOS]\label{alg:kwos}
Fix a small parameter $\varepsilon>0$.
\begin{enumerate}
\item
Initialize:  $w_x(0)=x$. 
\item
While $\dist(w_{x}(j), D) > \varepsilon$:
\begin{enumerate}
\item
Set $r(j) = \dist(w_{x}(j),\partial D)$.
\item
Kill the particle with probability $1-\psi(\lambda r(j)^2)$.
If the particle is killed, the algorithm terminates and returns $0$.
\item
Sample $\xi(j)$ independently from the unit sphere $\partial B_{n}(0,1)$
(by using Algorithm \ref{alg:uniform-sphere}). 
\item
Set $w_{x}(j+1) = w_{x}(j) + r(j)\xi(j)$.
\end{enumerate} 
\item
When $\dist(w_{x}(j), \partial D) \le \varepsilon$:
\item
Set $\mathrm{pr}\, w_{x}(j)$ to be the orthogonal projection of $w_{x}(j)$ to $\partial D$.
\item
Return $\mathrm{pr}\, w_{x}(j)$.
\end{enumerate}
\end{alg}

\section{Examples and comparisons}\label{sect:comparisons}

In this section, we give examples to motivate our algorithms and to illustrate their potential applications. The examples were computed by using a straightforward implementation algorithms in one and two-dimensional settings and chosen from the point of view of visualization. All computations were performed on a MacBook Air laptop. Wolfram Mathematica 10.2 and basic implementations of the algorithms with no performance optimizations were used. 

Obviously, the stochastic approaches presented here are more attractive in higher dimensions, where many deterministic simulation methods  are not available, or lead into excessive computation times. Also, it should be noted that the algorithms in this paper are particularly suitable for parallel computation, as simulated paths are independent from each other.

\subsection*{One-dimensional example}
Let $D=[0,1]$ and consider the differential equation $u''(t)=-4u(t)$ with boundary values $u(0)=1$ and $u(1)=3$. Then the exact solution to the boundary value problem is given by
\begin{equation}
u(t)=\cos(2t)-\cot(2)\sin(2t)+3\csc(2)\sin(2t).
\end{equation}
The exact solution $u$ as well as its approximations with the DWOS algorithm, where the problem is first lifted into dimension two by using the Duffin correspondence, and with the WWOS algorithm are illustrated in Figure \ref{fig:onedim}.

\begin{figure}[h]

\includegraphics[width=5.5cm]{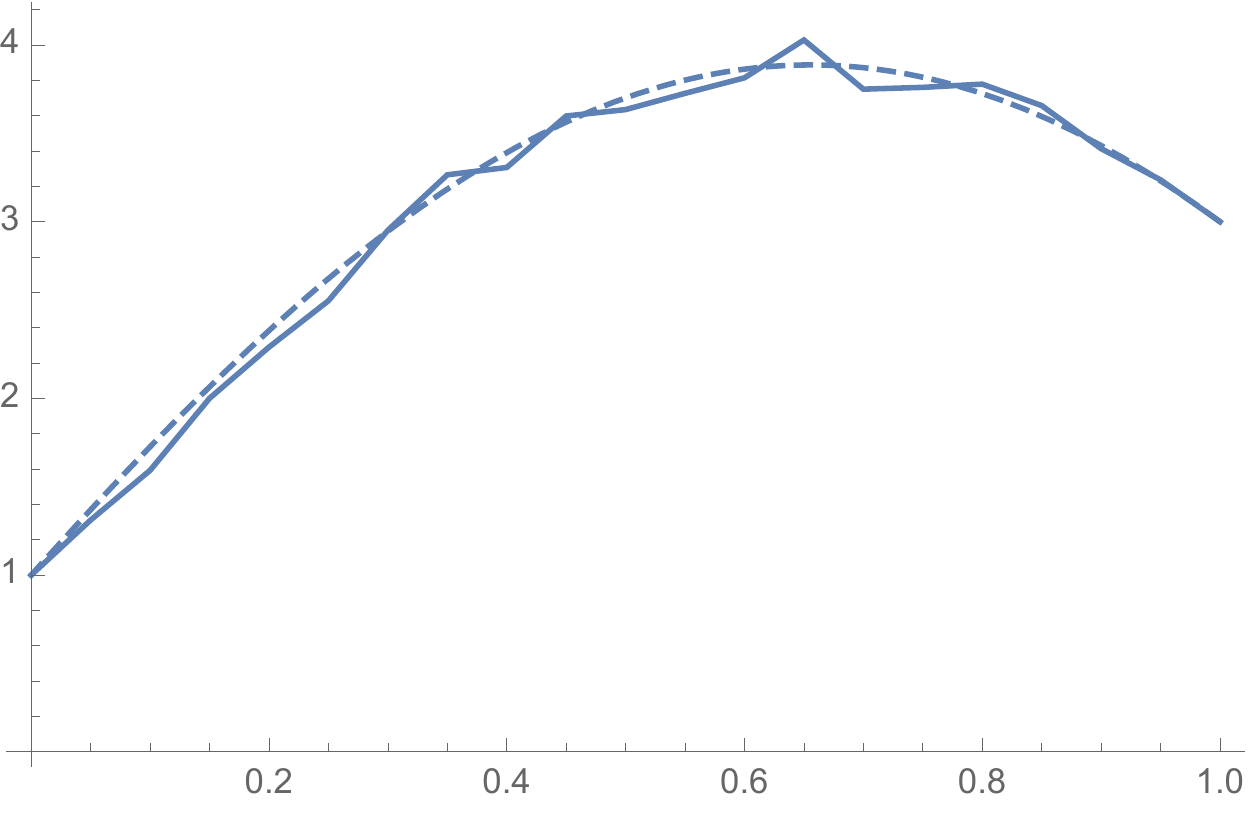}\qquad
\includegraphics[width=5.5cm]{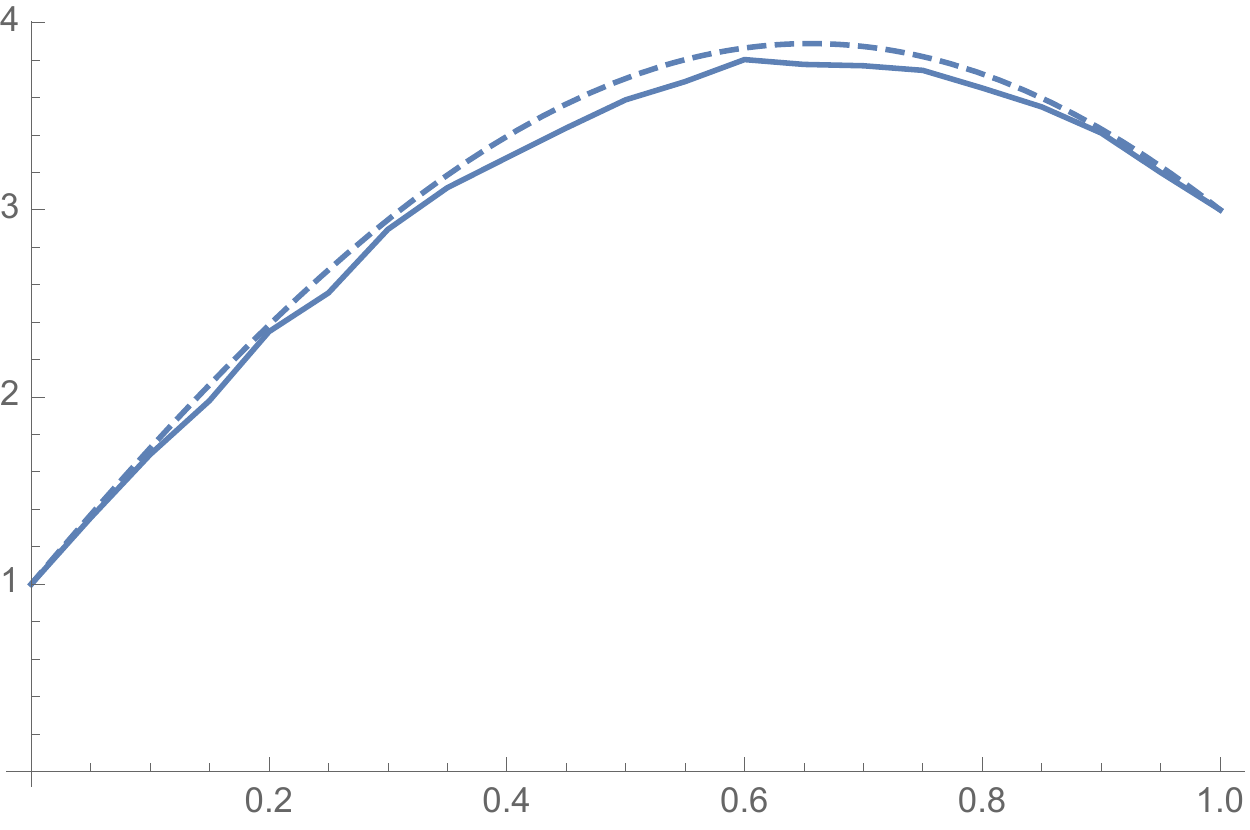}

\caption{ The exact solution $u$ (dashed) and its approximations $\hat u_{10\,000}$ with DWOS (left) and WWOS (right) algorithms. }\label{fig:onedim}
\end{figure}

\subsection*{Two-dimensional examples}
Next, we consider two simple boundary value problems of the equation \eqref{eq:helmholtz-yukawa} on polygonal domains in the plane:
\begin{enumerate}
\item The trapezoidal domain $D_1$ defined by the points $(0,0), (1,0), (1,1)$ and $(0,2)$, where the boundary values are given by $f_1(x,y)=x^3+x^2$ and with $\lambda=-2$ (cf. Example 4.4 and Figure 2 of \cite{yrs}).
\item The non-convex L-shaped domain $D_2$ defined by the points $(0,0), (2,0), (2,1), (1,1), (1,2)$ and $(0,2)$, where boundary values are given by the non-continuous function $f_2(x,y)=\1_{\{(x,y) : x,y\neq 0\}}$ and $\lambda =-1$.
\end{enumerate}
We compute an approximation $\hat u$ to the solution of the above boundary value problems  by using both the DWOS algorithm and the WWOS  algorithm. For DWOS, the problem is lifted to the dimension three.  The DWOS and WWOS on the domain $D_2$ are illustrated in Figure \ref{fig:wos23dimL}. Approximations to the solutions are illustrated in Figure \ref{fig:sol-trape} and Figure \ref{fig:sol-L}. Our experiments suggest that WWOS algorithm is more stable than DWOS and thus a smaller number of simulations are required for a comparable result.

\begin{figure}[h]

\includegraphics[width=5.5cm]{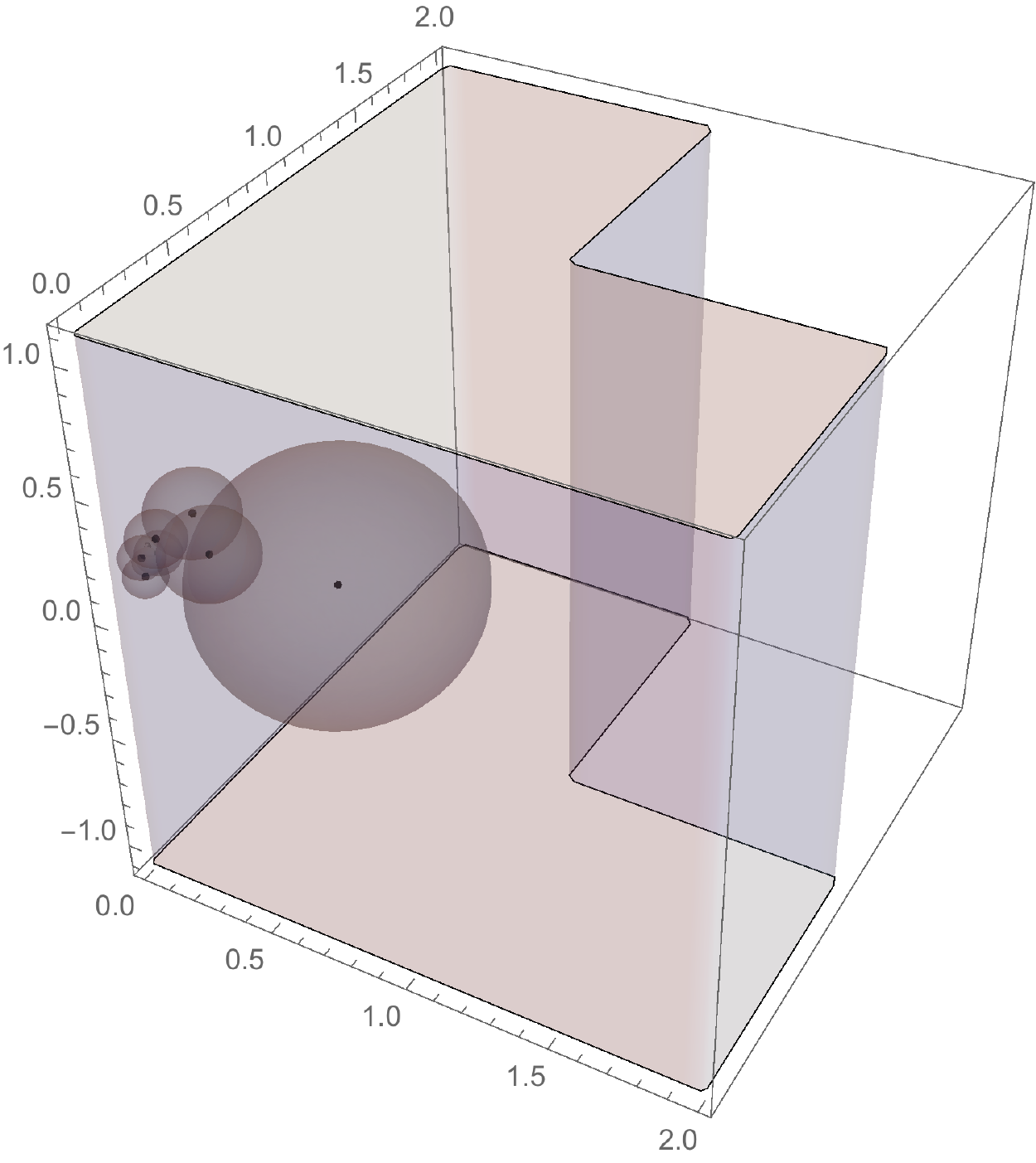}\qquad
\includegraphics[width=5.5cm]{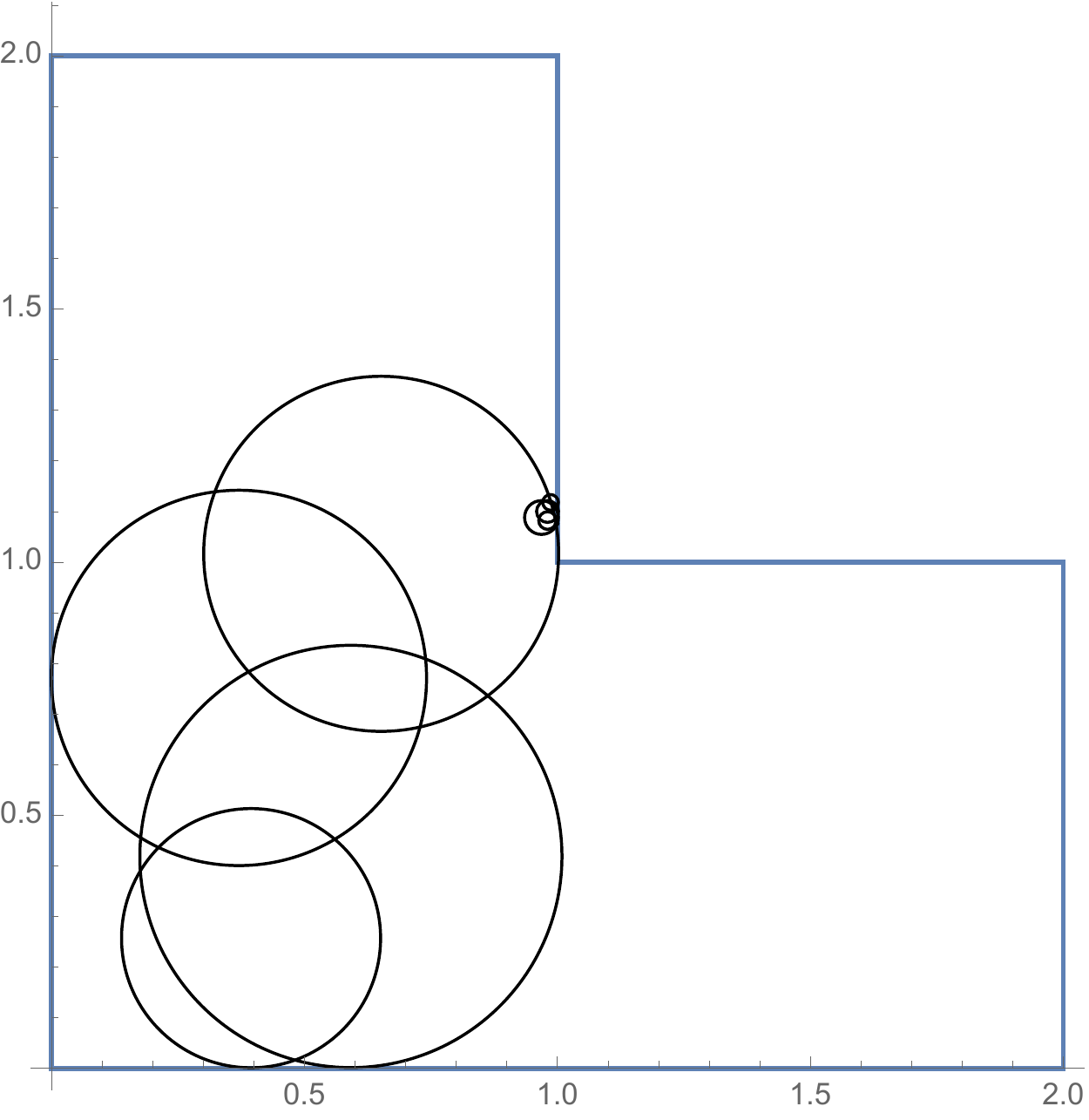}

\caption{Simulation of an individual path on the domain $D_2$ with DWOS (left) and WWOS (right) algorithms. In DWOS, the problem is first lifted into the dimension three by using the Duffin correspondence. The WWOS algorithm is purely two-dimensional, but it requires computation of the weight function $\psi$. }\label{fig:wos23dimL}
\end{figure}

\begin{figure}[h]

\includegraphics[width=5.8cm]{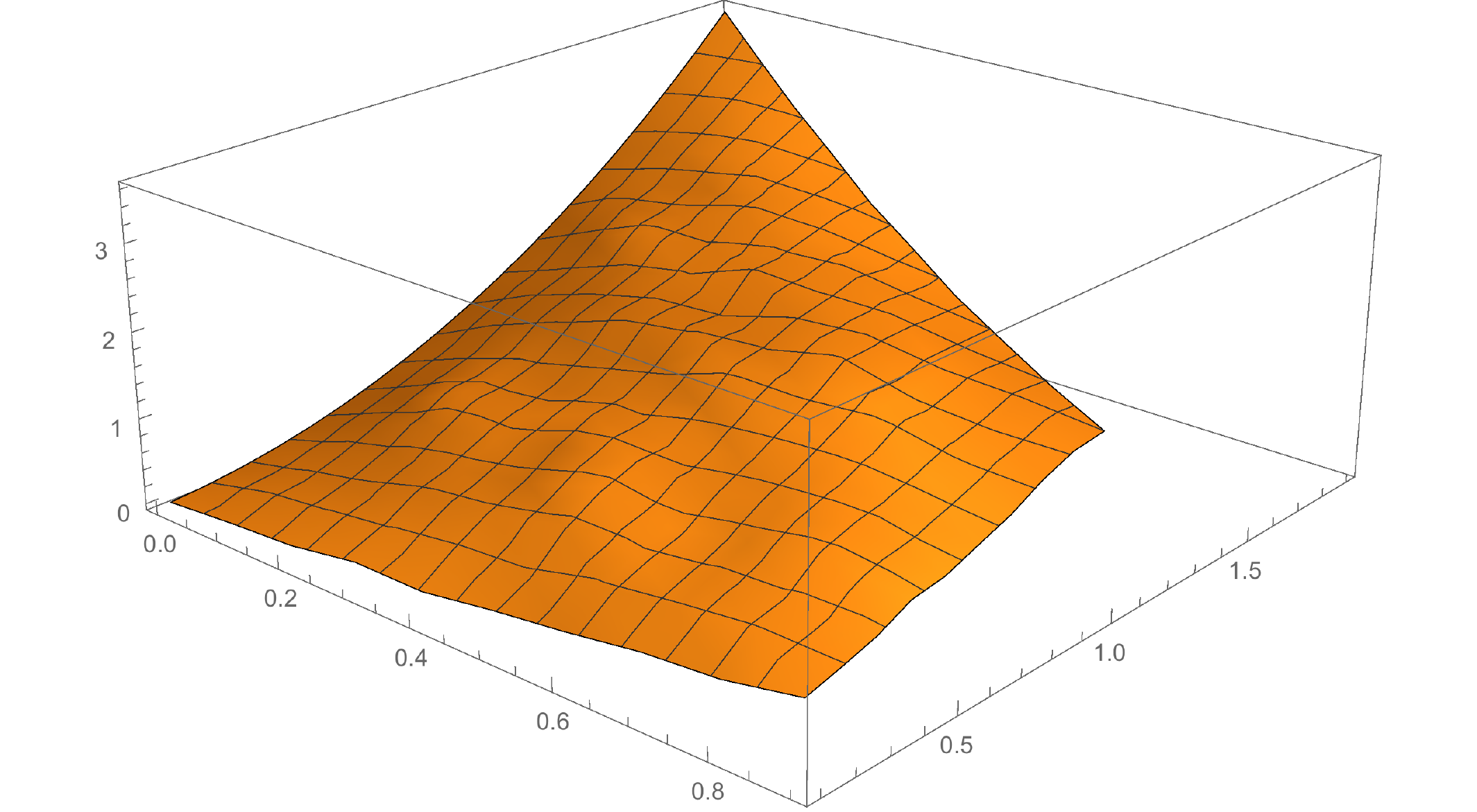}\qquad
\includegraphics[width=5.2cm]{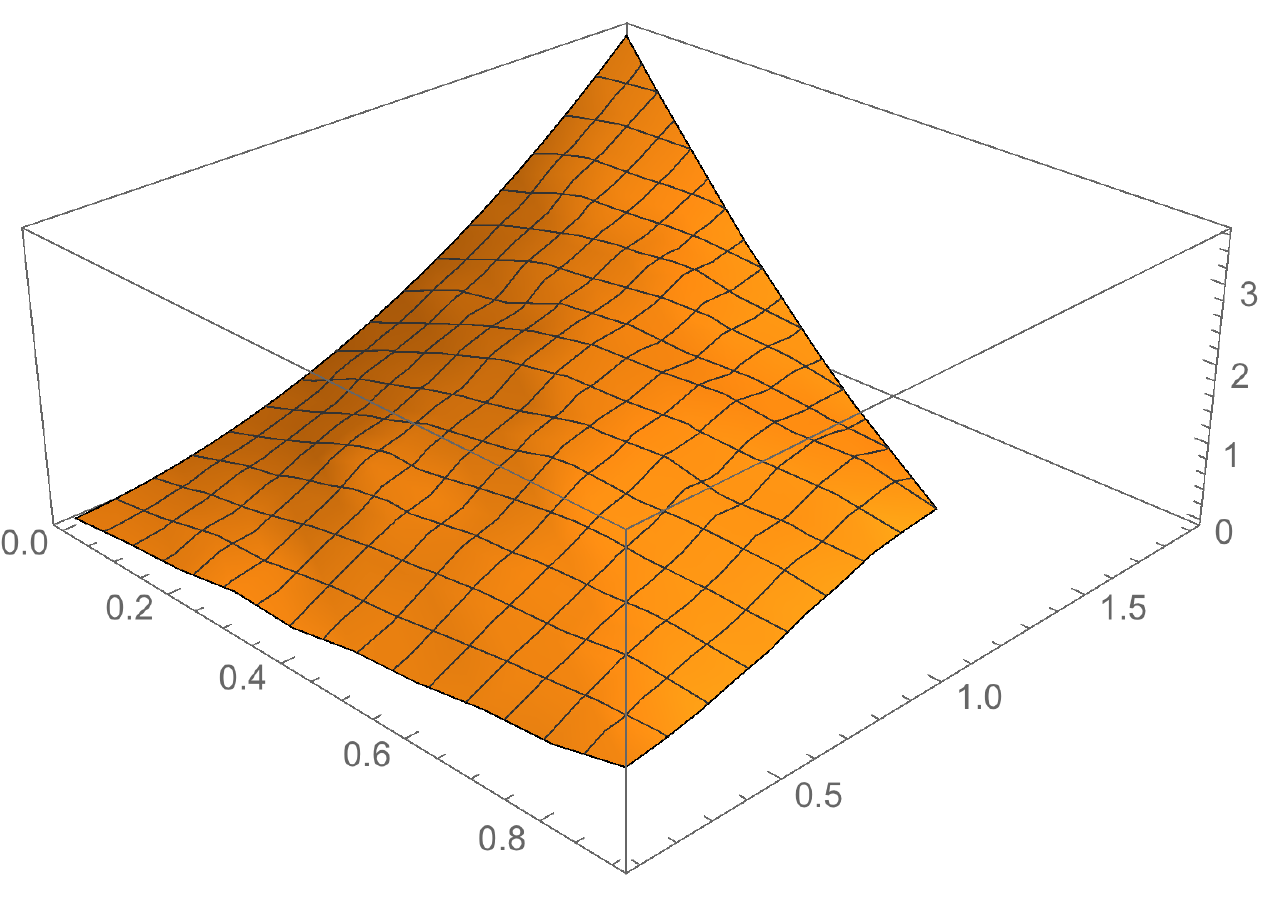}

\caption{Simulated approximations $\hat u_{1000}$ (DWOS, left) and $\hat u_{500}$ (WWOS, right) and  of the solution on the domain $D_1$. The boundary values are given by the function $f_1$ and $\lambda=-2$. }\label{fig:sol-trape}
\end{figure}

\begin{figure}[h]

\includegraphics[width=5.8cm]{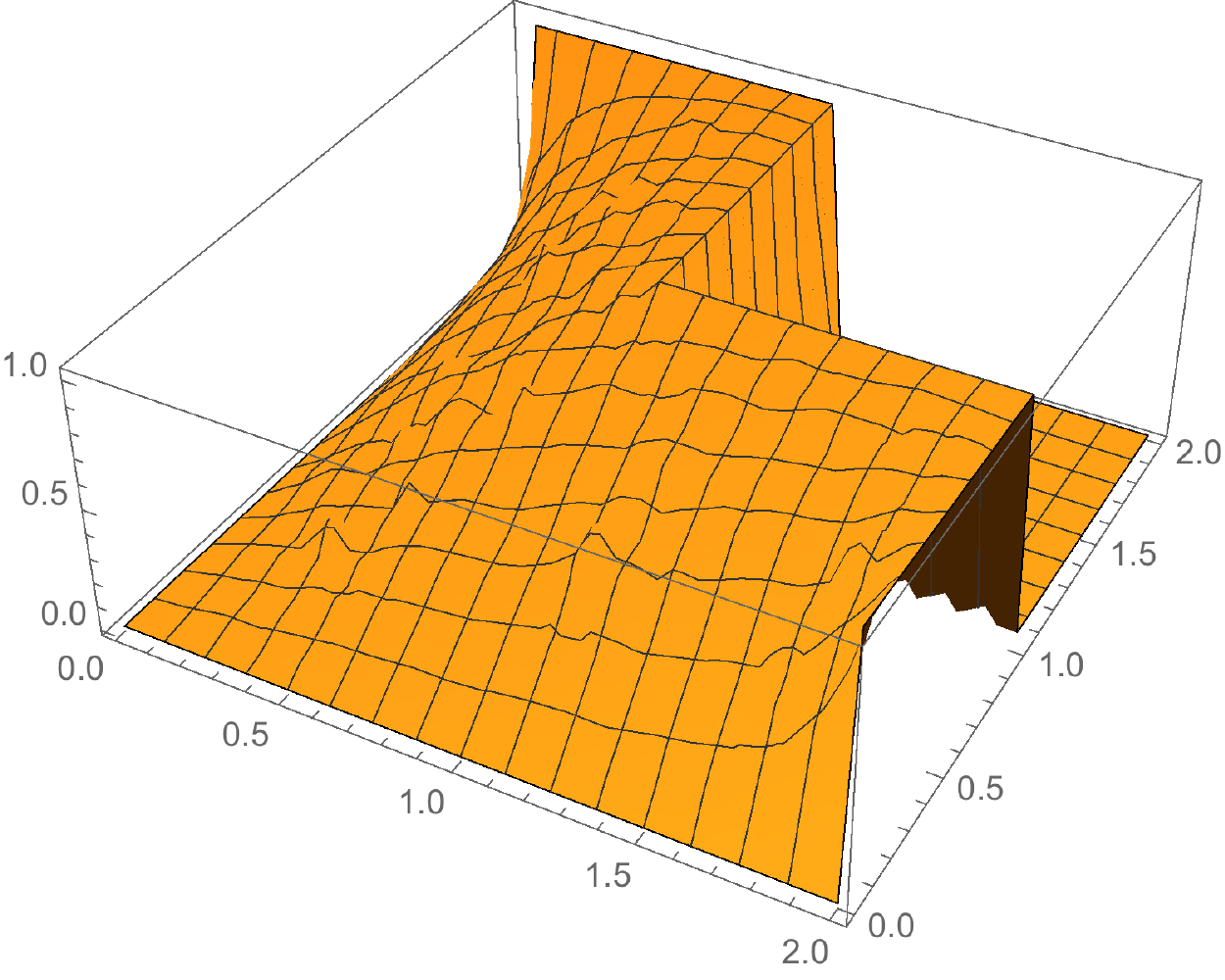}\qquad
\includegraphics[width=5.2cm]{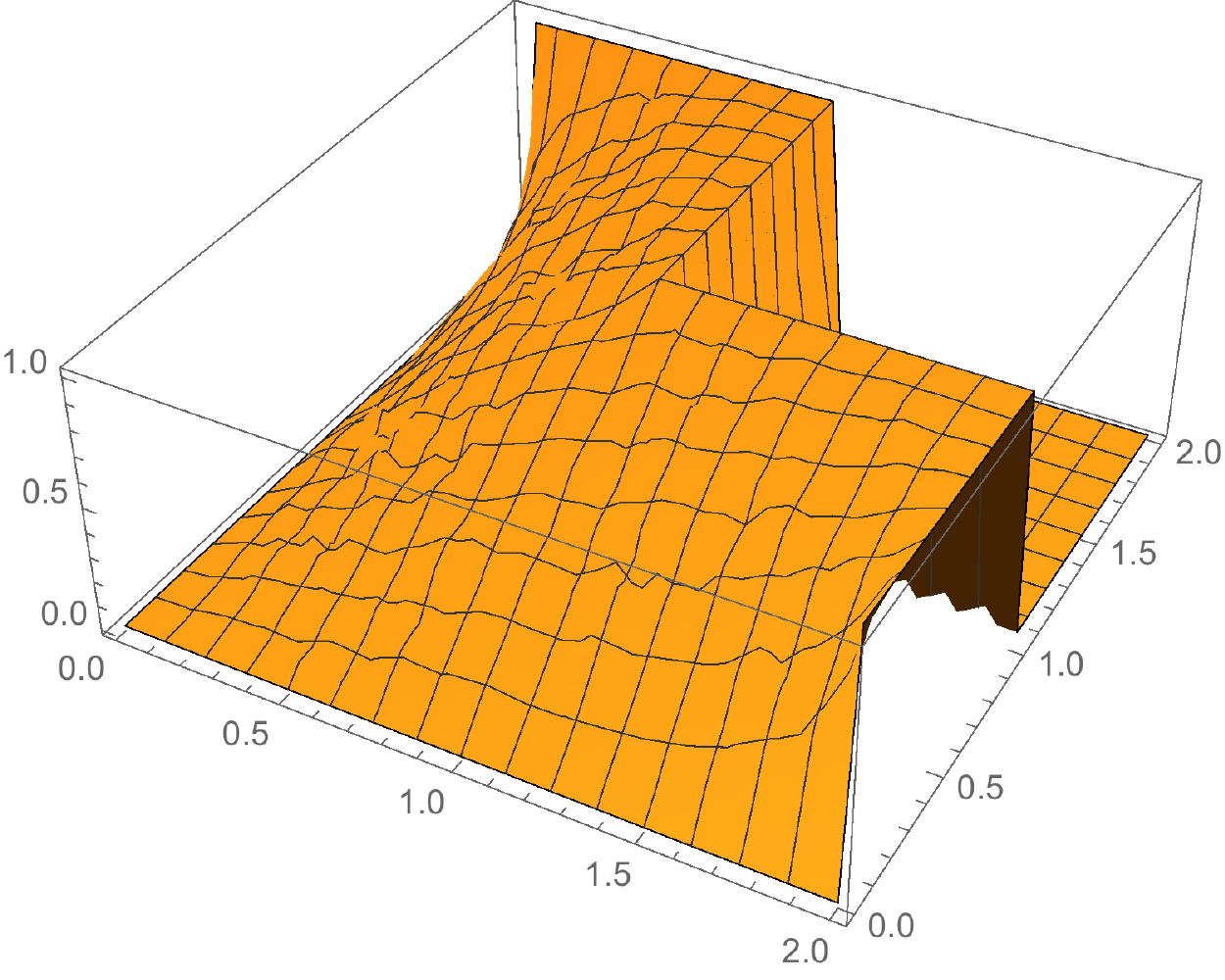}

\caption{Simulated approximations $\hat u_{1000}$ (DWOS, left) and $\hat u_{500}$ (WWOS, right) and  of the solution on the domain $D_2$. The boundary values are given by the function $f_2$ and $\lambda=-1$. }\label{fig:sol-L}
\end{figure}

\section{Conclusion}\label{sect:conclusions}

It is clear that the KWOS algorithm is better than the WWOS algorithm, when it is applicable (i.e., in the Yukawa case).  Our experiments in Section \ref{sect:comparisons} suggest that, at least on low dimensions, WWOS algorithm is more stable than the DWOS algorithm.  Consequently, it seems that the DWOS algorithm is best used in high dimensions, where adding one extra dimension should not make much difference.  However, the DWOS algorithm is an extension, not a modification, of the classical WOS algorithm, i.e., if one has WOS already implemented, implementing DWOS is simply a matter of giving different input parameters to the WOS algorithm.


\begin{thebibliography}{22}







\bibitem{chung-zhao}
{\sc K.L. Chung} and {\sc Z. Zhao}: {\it From Brownian motion to Schr\"odinger's equation}. 2nd Printing. Springer, 2001.

\bibitem{ciesielski-taylor}
{\sc Z. Ciesielski} and {\sc S.J. Taylor}: First passage times and sojourn times for Brownian motion in space and the exact Hausdorff measure of the sample path. {\it Trans. Amer. Math. Soc.} {\bf 103} (1962), 434--450.


\bibitem{woms}
{\sc M. Deaconu} and {\sc S. Herrmann}: Hitting Time for Bessel Processes --- Walk on Moving Spheres Algorithm (WOMS). {\it Ann. Appl. Probab.} {\bf 23}, No. 6 (2013), 2259--2289.

\bibitem{woms2}
{\sc M. Deaconu}, {\sc S. Herrmann} and {\sc S. Maire}: The walk on moving spheres: A new tool for simulating Brownian motion's exit time from a domain, {\it Math. Comput. Simulation} (2015), article in press.



\bibitem{duffin1} {\sc R.J. Duffin}: Yukawan potential theory. {\it J. Math. Anal. Appl.} {\bf 35} (1971), 105--130.


\bibitem{durrett} {\sc R. Durrett}: {\it Stochastic Calculus: A Practical Introduction}. CRC Press, 1996.

\bibitem{elepov-mikhailov}
{\sc B.S.~Elepov} and {\sc G.A.~Mihailov}: The ``Walk On Spheres'' algorithm for the equation $\Delta u - cu = -g$. {\it Soviet Math. Dokl.} {\bf 14}
(1973) 1276--1280.





\bibitem{hwang-mascagni}
{\sc C.-O.~Hwang} and {\sc M. Mascagni}: Efficient modified ``Walk On Spheres''
algorithm for the linearized Poisson-Boltzmann equation. 
{\it Appl. Phys. Lett.} {\bf 78}(6) (2001), 787--789.

\bibitem{hwang-mascagni-given}
{\sc C.-O.~Hwang}, {\sc M.~Mascagni} and {\sc J.A.~Given}: A Feynman--Kac path-integral implementation for Poisson's equation using an $h$-conditioned Green's function. {\it Math. Comput. Simulation} {\bf  62} (2003), 347--355.







\bibitem{kent} {\sc J.T.~Kent}: Eigenvalue expansion for diffusion hitting times. {\it Z. Wahr. Ver. Gebiete} {\bf 52} (1980), 309--319.


\bibitem{kakutani} {\sc S. Kakutani}: On Brownian motion in $n$-space. {\it Proc. Imp. Acad. Tokyo} {\bf 20}(9): 648–652, 1944.

\bibitem{krahn} {\sc E.~Krahn}: \"Uber Minimaleigenschaft der Kugel in drei und mehr Dimensionen. {\it Acta Comm. Univ. Tartu (Dorpat)} {\bf A9} (1926), 1--44.



\bibitem{muller} {\sc M.E. Muller}: Some continuous Monte Carlo methods for the Dirichlet problem. {\it Ann. Math. Statist.} {\bf 27} (1956), 569--589.


\bibitem{panharmonic} {\sc A. Rasila} and {\sc T. Sottinen}: Yukawa Potential, Panharmonic Measure and Brownian Motion. Preprint. arXiv:1310.2167 (2015)


\bibitem{wendel} {\sc J.G. Wendel}: Hitting Spheres with Brownian Motion. {\it Ann. Probab.} {\bf 8}(1), (1980), 164--169.

\bibitem{yrs} {\sc X. Yang}, {\sc A. Rasila} and {\sc T. Sottinen}:
Efficient simulation of Schrödinger equation with piecewise constant positive potential. Preprint. arXiv:1512.01306. (2015)


\end{thebibliography}
\end{document}